\documentclass{amsart}
\usepackage{color}
\usepackage{amsmath}
\usepackage{amsfonts}
\usepackage{amssymb}
\usepackage{graphicx, enumitem}
\usepackage{hyperref}
\usepackage{tabularx}
\def \rank{{\rm rank}}

\def \N{\mathbb{N}}

\def \R{\mathbb{R}}

\def \G{\mathcal{G}}
\def \A{\mathcal{A}}

\def \G{\mathcal G}

\def \int{{\rm int\,}}

\def \U{\mathcal U}

\def \G{\mathcal G}

\def \int{{\rm int\,}}

\def \U{\mathcal U}

\providecommand{\U}[1]{\protect \rule{.1in}{.1in}}
\newtheorem{theorem}{Theorem}[section]

\newtheorem{corollary}[theorem]{Corollary}

\newtheorem{definition}[theorem]{Definition}

\newtheorem{lemma}[theorem]{Lemma}
\newtheorem{notation}[theorem]{Notation}

\newtheorem{proposition}[theorem]{Proposition}

\begin{document}
	\title{On the explicit formula for Gauss-Jordan elimination}
	\author{Nam Van Tran, Julia Justino and Imme van den Berg}
	\address{namtv@hcmute.edu.vn; julia.justino@estsetubal.ips.pt; ivdb@uevora.pt }

			\begin{abstract} The elements of the successive intermediate matrices of the Gauss-Jordan elimination procedure have the form of quotients of minors. Instead of the proof using identities of determinants of \cite{Li}, a direct proof by induction is given. 
		\end{abstract}
		\maketitle



		\section{Introduction}
		
		Gantmacher's book \cite{Gantmacher} on linear algebra contains an explicit formula for all elements $ a_{ij}^{(k)} $
		of the intermediate matrix obtained after $ k $ Gaussian operations applied to a matrix $ A=[a_{ij}] $, below and to the right of the $ k^{th} $ pivot. The formula is given in terms of quotients of minors, and follows when applying Gaussian elimination to the two minors, which happen to have common factors all wiping themselves out, except for $ a_{ij}^{(k)} $. Up to changing indices, the formula holds also for Gauss-Jordan elimination, and then a similar formula, with alternating sign, holds at the upper part of the intermediate matrices. A proof is given by Y. Li in \cite{Li}, using identities of determinants. To our opinion the notation used for minors in different parts of the matrix is somewhat confusing. We considered it worthwhile to present the result in the notation of \cite{Gantmacher}, with a direct proof by induction. Indeed, up to some changes, it is possible to extend the method of simultaneous Gaussian elimination of minors to all elements of the $ k^{th} $ intermediate matrix.
		
		The status of the explicit formula for the Gauss-Jordan elimination procedure seems to be uncertain. As regards to the formula for the lower part of the intermediate matrices Gantmacher refers to \cite{Grossman}; \cite{Scott} contains some historical observations and presents a proof using identities of determinants.  
		
		In \cite{Scott} the explicit formula for Gaussian elimination was applied to error analysis. We came across the explicit formula for Gauss-Jordan elimination also in relation to error analysis \cite{Imme jus and Nam}. Here we apply the formula to prove that principal minors satisfying a maximality property are non-zero. This has also some numerical relevance, for a consequence is that maximal pivots are automatically non-zero.

		\section{The explicit formula for Gauss-Jordan elimination}\label{sectionGJ}
		We start with some definitions and notations related to the Gauss-Jordan operations, where we use the common representation by matrix multiplications. It is convenient when the matrix is diagonally eliminable, i.e. all pivots lie on the principal diagonal, and we verify that this can be assumed without restriction of generality through a condition on minors, and that then the pivots are non-zero indeed. Theorem \ref{cmcttruyhoi} is the main theorem and gives explicit expressions for the intermediate matrices of the Gauss-Jordan procedure. Explicit formulas for the matrices representing the Gauss-Jordan operations follow directly and are given in Theorem \ref{Gmat}. 
		
		Theorem \ref{diagonally} is a straightforward consequence of Theorem \ref{cmcttruyhoi}, and states, together with Proposition \ref{changematran moi},  that without restriction of generality we may impose a maximality condition on the principal minors, and then the pivots are also maximal, and nonzero indeed.
		
		We will always consider $ m\times n $ matrices with $m,n\in \mathbb{N},m,n\geq 1$.	We denote by $M_{m, n}(\mathbb{R})$ the set of all $m\times n$
		matrices over the field $\mathbb{R}$.

		\begin{definition}\label{defgau}
			Let $A=[a_{ij}]_{m\times n}\in M_{m,n}(\mathbb{R})$  be of \rank $r\geq 1$. Assume that $a_{11}\not=0$.  We let 
			$\mathcal{G}_{1}=
			\begin{bmatrix}
				g_{ij}^{(1)}%
			\end{bmatrix}%
			_{m\times m}  $
			be the matrix which corresponds to the multiplication of the entries of the first line of $ A $ by $1/ a_{11} $, such that the first pivot of $A^{(1)}\equiv \mathcal{G}_{1} A=[a^{(1)}_{ij}]_{m\times n}$ becomes $ a^{(1)}_{11} =1$. This means that 	
			\begin{equation*}
				\begin{array}{l}
					\mathcal{G}_{1}=%
					\begin{bmatrix}
						g_{ij}^{(1)}%
					\end{bmatrix}%
					_{m\times m}=%
					\begin{bmatrix}
						\frac{1}{a_{11}} & 0 & \cdots & 0 \\ 
						0 & 1 & \cdots & 0 \\ 
						\vdots & \vdots & \ddots & \vdots \\ 
						0 & 0 & \cdots & 1
					\end{bmatrix}.%
				\end{array}
			\end{equation*}
			Assume $a^{(1)}_{22}\not=0$. Let $ \mathcal{G}_{2} $ be the matrix which corresponds to the creation of zero's in the first column of $ A^{(1)}$, except for $a_{11}^{(1)}$, and let $ A^{(2)}=\mathcal{G}_{2}A^{(1)}=[a^{(2)}_{ij}]_{m\times n}$, i.e.
			\begin{equation*}
				\mathcal{G}_{2}=%
				\begin{bmatrix}
					g_{ij}^{(2)}%
				\end{bmatrix}%
				_{m\times m}=%
				\begin{bmatrix}
					1 & 0 & \cdots & 0 \\ 
					-a_{21} & 1 & \cdots & 0 \\ 
					\vdots & \vdots & \ddots & \vdots \\ 
					-a_{m1} & 0 & \cdots & 1%
				\end{bmatrix},
			\end{equation*}
			and 
			\begin{equation*}
				A^{(2)}=%
				\begin{bmatrix}
					a_{ij}^{(2)}%
				\end{bmatrix}%
				_{m\times n}=%
				\begin{bmatrix}
					1 & a^{(1)}_{12} & \cdots & a^{(1)}_{1n} \\ 
					0 & a^{(2)}_{22} & \cdots & a^{(2)}_{2n} \\ 
					\vdots & \vdots & \ddots & \vdots \\ 
					0 & a^{(2)}_{m2} & \cdots & a^{(2)}_{mn}%
				\end{bmatrix}.
			\end{equation*}	
			Assume  that $ \mathcal{G}_{2k}=[g^{(2k)}]_{m\times m}$ and $A^{(2k)}=[a^{(2k)}_{ij}]_{m\times n} $ are defined for $ k<r $, and $a^{(2k)}_{k+1k+1} \not=0$.  The matrix $ \mathcal{G}_{2k+1} $ corresponds to the multiplication of row $ k+1$ of $ A^{(2k)} $ by $1/ a^{(2k)}_{k+1k+1} $, leading to $ A^{(2k+1)}\equiv\mathcal{G}_{2k+1}A^{(2k)} $, and the matrix $ \mathcal{G}_{2k+2} $ corresponds to transforming the entries of column $ k $ of $ A^{(2k+1)} $ into zero, except for the entry $ a^{(2k+1)}_{k+1k+1}(=1) $, resulting in $ A^{(2k+2)}\equiv\mathcal{G}_{2k+2}A^{(2k+1)} $. So we have
			$ 
			\mathcal{G}_{2k+1}=%
			\begin{bmatrix}
				g_{ij}^{(2k+1)}%
			\end{bmatrix}%
			_{m\times m}\text{,\quad where}%
			$%
			\begin{equation}\label{ct g2k+1}
				g_{ij}^{(2k+1)}=%
				\begin{cases}
					\quad 1 & \mbox{ if }i=j\not=k+1 \\ 
					\quad 0 & \mbox{ if }i\not=j \\ 
					\dfrac{1}{a^{(2k)}_{k+1k+1}} & \mbox{ if }i=j=k+1%
				\end{cases}
			\end{equation}
			and $%
			\begin{array}{rr}
				\mathcal{G}_{2k+2}=%
				\begin{bmatrix}
					g_{ij}^{(2k+2)}%
				\end{bmatrix}%
				_{m\times m}\text{,\quad where} & 
			\end{array}%
			$
			\begin{equation}\label{ct g2k+2} 
				g_{ij}^{(2k+2)}=%
				\begin{cases}
					0 & \mbox{ if }j\not \in \{i,k+1\} \\ 
					1 & \mbox{ if }i=j \\ 
					-a^{(2k+1)}_{ik+1}
					\medskip  &  \mbox{ if }i\neq k+1,j=k+1%
				\end{cases}.
			\end{equation}	
			For $1\leq q\leq 2r$ we call the matrix $ A^{(q)} $ the \emph{$ q^{th} $ Gauss-Jordan intermediate matrix}, and the matrix $\G_{q} $ is called the \emph{$ q^{th} $ Gauss-Jordan operation matrix}. We write $ \mathcal{G}=\mathcal{G}_{2r}\mathcal{G}_{2r-1}\cdots\mathcal{G}_{1}$.
		\end{definition}
		The product $\mathcal{G}_{2k+1}A^{(2k)}$ corresponds to the Gaussian
		operation of multiplying the $(k+1)^{th}$ row of the matrix $A^{(2k)}$ by the non-zero
		scalar $\dfrac{1}{a^{(2k)}_{k+1k+1}}$ and the product $\mathcal{G}_{2k}A^{(2k-1)}$ corresponds to the repeated
		Gauss-Jordan operation of adding a scalar multiple of a row to some other row of $A^{(2k-1)}$.
		
		\begin{definition}
			\label{matrixofgaussJordanelimination1} Assume $A=[a_{ij}]_{m\times n}\in M_{m,n}(\mathbb{R}%
			) $ has \rank $r\geq 1$. Then $A$ is called 	\emph{diagonally eliminable up to $ r $} if $ a^{(2k-2)}_{kk}\not=0 $ for $1\leq k\leq r$; if $ r=m\leq n $ we say that $ A $ is \emph{diagonally eliminable}. 
		\end{definition}
		
		\begin{notation}
			\label{kyhieu2}Let $A\in M_{m,n}(\mathbb{R})$. For each $k\in \N $ such that $1\leq k\leq \min\{m, n\},$ let $1\leq i_{1}<\dots <i_{k}\leq m$ and $1\leq
			j_{1}<\dots <j_{k}\leq n$. 
			
			\begin{enumerate}
				\item We denote the $k\times k$ submatrix  of $A$ consisting of the rows with indices $\{i_{1},\dots ,i_{k}\}$ and columns with indices $\{j_{1},\dots ,j_{k}\}$ by $A^{i_{1}\dots i_{k}}_{j_{1}\dots j_{k}}$. 
				
				\item We denote the corresponding $k\times k$ minor by $m^{i_{1}\dots i_{k}}_{j_{1}\dots j_{k}}=\det\left( A^{i_{1}\dots i_{k}}_{j_{1}\dots j_{k}}\right) $.

				\item \textrm{For $ 1\leq k\leq \min\{m,n\}$ we may denote the principal minor of order $ k $ by $m_{k}=m_{1\cdots k}^{1\cdots k}$. We define formally $m_{0}=1$}.
			\end{enumerate}
		\end{notation}
		
		Let $ A $ be a matrix of \rank $ r\geq 1 $. It is well-known and not difficult to see that up to changing rows and columns one may always assume that the principal minors up to $ r $ are all non-zero. Proposition \ref{propa2k} shows that this condition is equivalent to being diagonally eliminable, and gives also a formula for the pivots. The proposition is a consequence of the following lemma; its proof uses the idea found in \cite{Gantmacher}, on how the value of certain elements of the intermediate matrices can be related to minors by simplifying determinants. 
		\begin{lemma}\label{lemmkprod}Let $A=[a_{ij}]_{m\times n}\in M_{m,n}(\mathbb{R}) $ be of \rank $ r> 1$. Assume that  $ a_{11},a^{(1)}_{22},\dots, a^{(2k-2)}_{kk}$  are non-zero for $ 1\leq k<r $. Then for $1\leq k<r$ it holds that
			\begin{equation}\label{mkprod}
				m_{k+1}=a_{11}a^{(2)}_{22}\cdots a^{(2k-2)}_{kk}a^{(2k)}_{k+1k+1}.
			\end{equation}
			As a consequence $a_{11},a^{(1)}_{22},\dots, a^{(2k-2)}_{kk}, a^{(2k)}_{k+1k+1}\not=0$ if and only if $m_1,\dots, m_{k+1}\not=0$.  
		\end{lemma}
		
		\begin{proof}  Assume that $  1\leq k<r $ and  that $ a_{11}, a^{(1)}_{22},\dots, a^{(2k-2)}_{kk}$ are all non-zero. Then we may apply the Gauss-Jordan operations up to $ 2k $ and obtain
			\begin{align*}
				m_{k+1}=&\det
				\begin{bmatrix}
					a_{11} & \cdots & a_{1k} & a_{1k+1} \\ 
					\vdots & \ddots & \vdots & \vdots \\ 
					a_{k1} & \cdots & a_{kk} & a_{kk+1} \\ 
					a_{k+11} & \cdots & a_{k+1k} & a_{k+1k+1}%
				\end{bmatrix}\\
				=&a_{11}\cdots a^{(2k-1)}_{kk}\det\begin{bmatrix}
					1 & \cdots & 0 & a^{(2k-1)}_{1k+1} \\ 
					\vdots & \ddots & \vdots & \vdots \\ 
					0 & \cdots & 1 & a^{(2k-1)}_{kk+1} \\ 
					0 & \cdots & 0 & a^{(2k)}_{k+1k+1}%
				\end{bmatrix}.%
			\end{align*}
			The Laplace-expansion applied to the last row yields 	
			$$ m_{k+1}=a_{11}\cdots a^{(2k-1)}_{kk}a^{(2k)}_{k+1k+1}\det 
			\begin{bmatrix}
				1 & \cdots & 0  \\ 
				\vdots & \ddots & \vdots  \\ 
				0 & \cdots & 1 		
			\end{bmatrix} =a_{11}\cdots a^{(2k-1)}_{kk}a^{(2k)}_{k+1k+1}. $$
			Using induction and formula \eqref{mkprod}  we derive the last part of the Lemma.
		\end{proof}
		
		\begin{proposition}\label{propa2k}
			Let $A=[a_{ij}]_{m\times n}\in M_{m,n}(\mathbb{R}) $ be of \rank $ r>1$. Then $A$ is  diagonally eliminable up to $r$ if and only if $m_1,\dots, m_r\neq 0$. In both cases, for $ 1\leq k<r $ it holds that 
			\begin{equation}\label{a2k}
				a^{(2k)}_{k+1k+1}=\frac{m_{k+1}}{m_{k}}. 
			\end{equation}
		\end{proposition}
		\begin{proof}
			Applying Lemma  \ref{lemmkprod} with $k=r-1$  we derive that $A$ is  diagonally eliminable up to $r$ if and only if $m_1,\dots, m_r\neq 0$. Then \eqref{a2k} follows from \eqref{mkprod} applied to $ k+1 $ and $ k $.
		\end{proof}
		Next theorem is the main theorem and gives formulas for the entries 
		$ a^{(2k)}_{ij}$ of the matrices  $A^{(2k)}$. The formulas are similar to \eqref{a2k} below to the right of the pivots, and above to the right they come with alternating sign. Again they are proved by simplifying determinants. 
		\begin{theorem}[Explicit expressions for the Gauss-Jordan intermediate matrices]
			\label{cmcttruyhoi} Let 
			$A= [a_{ij}]_{m\times n}\in M_{m,n}(\mathbb{R})$ be of \rank $ r>1$ and  diagonally eliminable up to $r$. Let $k<r$. Then 
			\begin{equation*}
				A^{(2k)}=%
				\begin{bmatrix}
					1 & \cdots  & 0 & a_{1k+1}^{(2k)} & \cdots  & a_{1n}^{(2k)} \\ 
					\vdots  & \ddots  & \vdots  & \vdots  & \ddots  & \vdots  \\ 
					0 & \cdots  & 1 & a_{kk+1}^{(2k)} & \cdots  & a_{kn}^{2k} \\ 
					0 & \cdots  & 0 & a_{k+1k+1}^{(2k)} & \cdots  & a_{k+1n}^{(2k)} \\ 
					\vdots  & \ddots  & \vdots  & \vdots  & \ddots  & \vdots  \\ 
					0 & \cdots  & 0 & a_{mk+1}^{(2k)} & \cdots  & a_{mn}^{(2k)}%
				\end{bmatrix},%
			\end{equation*}%
			where 
			\begin{equation}
				a_{ij}^{(2k)}=%
				\begin{cases}
					(-1)^{k+i} \dfrac{m^{1\dots k}_{1\dots i-1i+1\dots kj}}{m_{k}}%
					\medskip  & \mbox{ if }1\leq i\leq k, k+1\leq j\leq n \\ 
					
					\qquad \qquad \dfrac{m^{1\dots ki}_{1\dots kj}}{m_{k}} & \mbox{ if }k+1\leq i\leq m,k+1\leq j\leq n
				\end{cases}.%
				\label{congthucquynapcuaA}
			\end{equation}%
		\end{theorem}
		\begin{proof}
			Firstly, let $ k+1\leq i\leq m$ and $k+1\leq j\leq n$. Let 
			\begin{equation*}
				U_{i,j}=
				\begin{bmatrix}
					a_{11} & \cdots & a_{1k} & a_{1j} \\ 
					\vdots & \ddots & \vdots & \vdots \\ 
					a_{k1} & \cdots & a_{kk} & a_{kj} \\ 
					a_{i1} & \cdots & a_{ik} & a_{ij}%
				\end{bmatrix}.%
			\end{equation*}
			Then $\det \left(U_{i,j}\right)=m^{1\dots ki}_{1\dots kj}.$  By applying the first $ 2k $ Gauss-Jordan operations to $ U_{i,j} $, we obtain
			\begin{equation*}
				\det \left(U_{i,j}\right)=a_{11}\cdots a^{(2k-2)}_{kk}\det\begin{bmatrix}
					1 & \cdots & 0 & a^{(2k-1)}_{1j} \\ 
					\vdots & \ddots & \vdots & \vdots \\ 
					0 & \cdots & 1 & a^{(2k-1)}_{kj} \\ 
					0 & \cdots & 0 & a^{(2k)}_{ij}%
				\end{bmatrix}=m_{k}a^{(2k)}_{ij}.
			\end{equation*}
			Hence $ a^{(2k)}_{ij}=\frac{m^{1\dots ki}_{1\dots kj}}{m_{k}} $.
			
			Secondly, we let $1\leq i<k+1$ and $k+1\leq j\leq n$. Let
			\begin{equation*}
				V_{i,j}=%
				\begin{bmatrix}
					a_{11} & \cdots & a_{1i-1} & a_{1i+1} & \cdots & a_{1k} & a_{1j} \\ 
					\vdots & \ddots & \vdots & \vdots & \ddots & \vdots & \vdots \\ 
					a_{i-11} & \cdots & a_{i-1i-1} & a_{i-1i+1} & \cdots & 
					a_{i-1k} & a_{i-1j} \\ 
					a_{i1} & \cdots & a_{ii-1} & a_{ii+1} & \cdots & a_{ik} & a_{ij} \\ 
					a_{i+11} & \cdots & a_{i+1i-1} & a_{i+1i+1} & \cdots & 
					a_{i+1k} & a_{i+1j} \\ 
					\vdots & \ddots & \vdots & \vdots & \ddots & \vdots \vdots &  \\ 
					a_{k1} & \cdots & a_{ki-1} & a_{ki+1} & \cdots & 
					a_{kk} & a_{kj}%
				\end{bmatrix}.%
			\end{equation*}
			Then
			\begin{equation*}
				\det(V_{i,j})=m^{1 \dots k}_{1\dots i-1i+1\dots kj}.
			\end{equation*}
			Let $ V_{i,j}^{\prime } $ be the matrix obtained by applying the first $ 2k $ Gauss-Jordan operations to $ V_{i,j} $. Then, using \eqref{mkprod},
			\begin{alignat*}{1}
				\det(V_{i,j} )=a_{11}\cdots  a^{(2k-2)}_{kk}\det
				\begin{bmatrix}
					1 & \cdots & 0 & 0 & \cdots & 0 & a^{(2k)}_{1j} \\ 
					\vdots & \ddots & \vdots & \vdots & \ddots & \vdots & \vdots \\ 
					0 & \cdots & 1 & 0 & \cdots & 0 & a_{i-1j}^{(2k)} \\ 
					0 & \cdots & 0 & 0 & \cdots & 0 & a_{ij}^{(2k)} \\ 
					0 & \cdots & 0 & 1 & \cdots & 0 & a_{i+1j}^{(2k)} \\ 
					\vdots & \ddots & \vdots & \vdots & \ddots & \vdots & \vdots \\ 
					0 & \cdots & 0 & 0 & \cdots & 1 & a_{kj}^{(2k)}%
				\end{bmatrix}=m_{k}\det(V_{i,j}^{\prime }).%
			\end{alignat*}
			Expanding $ \det(V_{i,j}^{\prime }) $ along the $i^{th}$ row, we derive that 
			\begin{equation*}  \label{dinh huc phu cong thuc gauss1}
				\det(V_{i,j}^{\prime })=(-1)^{i+k}a_{ij}^{(2k)}.
			\end{equation*}
			Combining, we conclude that $ a_{ij}^{(2k)}=
			(-1)^{i+k} \dfrac{m^{1\dots k}_{1\dots i-1i+1\dots kj}}{m_{k}} $. 
		\end{proof}	
		
		The next theorem gives explicit formulas for the matrices $ \G_{p} $ associated to the Gauss-Jordan operations. At odd order $ q=2k+1 $ we have to divide the row $ k+1 $ by $ a^{(2k)}_{k+1k+1} $ as given by \eqref{a2k}, and at even order $ q=2k+2 $, in the column $ j= k+1 $ we have to subtract by $ a^{(2k)}_{k+1j} $ as given by \eqref{congthucquynapcuaA}. 
		
		\begin{theorem}[Explicit expressions for the Gauss-Jordan operation matrices]\label{Gmat}Let \linebreak $A=[a_{ij}]_{m\times n}\in M_{m,n}(\mathbb{R}) $ be of \rank $ r>1$ and diagonally eliminable up to $ r $. For $ k<r $ the Gauss-Jordan operation matrix of odd order
			$ 	\mathcal{G}_{2k+1}=[	g_{ij}^{(2k+1)}]_{m\times m}$ 	satisfies
			\begin{equation}\label{Gmatodd}
				g_{ij}^{(2k+1)}=%
				\begin{cases}
					\quad 1 & \mbox{ if }i=j\not=k+1 \\ 
					\quad 0 & \mbox{ if }i\not=j \\ 
					\dfrac{m_{k}}{m_{k+1}} & \mbox{ if }i=j=k+1%
				\end{cases}.%
			\end{equation}
			and the Gauss-Jordan operation matrix of even order
			$
			\begin{array}{rr}
				\mathcal{G}_{2k+2}=%
				\begin{bmatrix}
					g_{ij}^{(2k+2)}%
				\end{bmatrix}%
				_{m\times m} 
			\end{array}%
			$ satisfies
			\begin{equation*}
				g_{ij}^{(2k+2)}=%
				\begin{cases}\qquad \qquad \qquad 0 & \mbox{ if }j\not \in \{i,k+1\} \\ 
					\qquad \qquad \qquad 1 & \mbox{ if }i=j \\ 
					(-1)^{k+i+1} \dfrac{m^{1\dots k}_{1\dots i-1i+1\dots k+1}}{m_{k}}%
					\medskip  & \mbox{ if }1\leq i\leq k,j= k+1 \\ 
					\qquad \qquad -\dfrac{m^{1\dots ki}_{1\dots kj}}{m_{k}} & \mbox{ if }k+1<i\leq m,j= k+1%
				\end{cases}.%
			\end{equation*}
			
		\end{theorem}
		
		\begin{proof}
			The theorem follows from formulas \eqref{ct g2k+1}, \eqref{ct g2k+2} and Theorem \ref{cmcttruyhoi}. 
		\end{proof} 
		
		
		Applying Theorem \ref{cmcttruyhoi} to a diagonally eliminable $ n\times n $ matrix we find at the end the identity matrix $I_n$. As a result the product of the Gauss-Jordan operation matrices is equal to the inverse matrix.
		\begin{corollary}
			Let $A=[a_{ij}]_{n\times n}$ be a  diagonally eliminable matrix. Then $\G A=I_n$ and $\G= A^{-1}$. 
		\end{corollary} 
		\begin{proof}
			By Theorem \ref{Gmat}, the matrices $\G_q$ are well-defined for all $1\leq q\leq 2n$. Then $\G A=A^{(2n)}=I_n$ by Theorem \ref{cmcttruyhoi}. Hence $\G=A^{-1}$.
		\end{proof} 
		
		Theorem \ref{cmcttruyhoi} holds under the condition that the matrix is diagonally eliminable, which is equivalent to asking that the principal minors are non-zero. Alternatively we may ask that the absolute values of the principal minors $m_{k+1}$ are maximal with respect to minors of the same size which share the first $ k $ rows and columns. It well-known that by appropriately changing rows and columns this may always be achieved, and then we speak of properly arranged matrices. We will use Theorem \ref{cmcttruyhoi} to show that the principal minors of properly arranged matrices are non-zero, which implies that they are diagonally eliminable. From a numerical point-of-view we are better off, the pivots being maximal.
		
		\begin{definition}
			\label{matrixofgaussJordanelimination} Assume $A=[a_{ij}]_{m\times n}\in M_{m,n}(\mathbb{R}%
			) $ has \rank $r\geq 1$. Then $A$ is called \emph{properly arranged}, if 
			\begin{equation} \label{dk2} 
				|a_{ij}|\leq |a_{11}| \mbox{ for all } 1\leq i\leq m, 1\leq j\leq n,
			\end{equation}  and for every $ k\in \N  $ such that $ 1\leq k< \min\{m,n\}$	
			\begin{equation}
				\left \vert m_{1\cdots k j}^{1\cdots ki}\right \vert \leq \left \vert m_{k+1}\right \vert \mbox{ for all } k+1\leq i\leq m, k+1\leq j\leq n.
				\label{maximumofprincipleminors}
			\end{equation}
		\end{definition}

		\begin{proposition}
			\label{changematran moi} Assume $A=[a_{ij}]_{m\times n}\in M_{m,n}(\mathbb{R}	) $ has \rank $r\geq 1$. By  changing rows and columns of $A$, if necessary, we may obtain that $ A $ is properly arranged.
		\end{proposition}

		\begin{theorem}\label{diagonally} 
			Let $A=[a_{ij}]_{m\times n}\in M_{m,n}(\R)$ be  of \rank $r\geq 1$ and properly arranged. Then $A$ is diagonally eliminable up to $r$.  Moreover, $|a^{(2k)}_{k+1k+1}|\geq \left|a^{(2k)}_{ij}\right|$ for all $k$ with   $0\leq k<r$ and $i,j$ with $k+1\leq i \leq m$ and  $k+1\leq j\leq n$.
		\end{theorem}
		\begin{proof}  
			Suppose  there exists $k$ with   $0\leq k<r$  such that $m_{k+1}=0$. We may also assume that $k$ is the smallest index satisfying this condition. If $m_{1}=0$, also $a_{11}=0$, and then by \eqref{dk2} all entries of $ A $ are zero. Hence $ \rank(A)=0 $, a contradiction. From now on we suppose that $ k\geq 1 $. By Theorem \ref{Gmat} the matrices $\G_1,\dots, \G_{2k}$ are well-defined.  By Theorem \ref{cmcttruyhoi} one has \begin{equation*}
				A^{(2k)}=%
				\begin{bmatrix}
					1 & \cdots  & 0 & a_{1k+1}^{(2k)} & \cdots  & a_{1n}^{(2k)} \\ 
					\vdots  & \ddots  & \vdots  & \vdots  & \ddots  & \vdots  \\ 
					0 & \cdots  & 1 & a_{kk+1}^{(2k)} & \cdots  & a_{kn}^{2k} \\ 
					0 & \cdots  & 0 & a_{k+1k+1}^{(2k)} & \cdots  & a_{k+1n}^{(2k)} \\ 
					\vdots  & \ddots  & \vdots  & \vdots  & \ddots  & \vdots  \\ 
					0 & \cdots  & 0 & a_{mk+1}^{(2k)} & \cdots  & a_{mn}^{(2k)}%
				\end{bmatrix},%
			\end{equation*}%
			where 
			\begin{equation*}
				a_{ij}^{(2k)}=%
				\begin{cases}
					(-1)^{k+i} \dfrac{m^{1\dots k}_{1\dots i-1i+1\dots kj}}{m_{k}}%
					\medskip  & \mbox{ if }1\leq i\leq k, k+1\leq j\leq n \\ 
					
					\qquad \qquad \dfrac{m^{1\dots ki}_{1\dots kj}}{m_{k}} & \mbox{ if } k+1\leq i\leq m, k+1\leq j\leq n
				\end{cases}.%
			\end{equation*}%
			Because $\A$ is properly arranged 
			\begin{equation}\label{akmax}
				\left|a^{(2k)}_{ij}\right|=\left|\dfrac{m^{1\dots ki}_{1\dots kj}}{m_{k}}\right|\leq \left| \dfrac{m_{k+1}}{m_k}\right|=\left|a^{(2k)}_{k+1k+1}\right|.
			\end{equation}	 
			for $k+1 \leq i\leq m, k+1\leq j\leq n$. Because $m_{k+1}=0$, it holds that $a^{(2k)}_{k+1k+1}=\dfrac{m_{k+1}}{m_k}=0$. Then $ a^{(2k)}_{ij}=0 $ for $k+1 \leq i\leq m, k+1\leq j\leq n$. Hence $\rank\left(A^{(2k)}\right)=k<r$. Then also  $\rank(A)=\rank\left(A^{(2k)}\right)=k<r$, a contradiction. Hence for $ 0\leq k<r $ one has $ m_{k+1} \neq 0$, and also \eqref{akmax}.
		\end{proof}
		\begin{corollary}
			Let $A=[a_{ij}]_{n\times n}\in M_{n}(\R)$ be  non-singular and properly arranged. Then $A$ is diagonally eliminable.
		\end{corollary}
		\begin{proof}
			Because $A=[a_{ij}]_{n\times n}$ is non-singular, it follows that $\rank(A)=n$. By Theorem \ref{diagonally} the matrix $A$ is diagonally eliminable. 
		\end{proof}

		\bigskip

		
	\end{document}